\documentclass[12pt,reqno]{amsart}
\usepackage{amssymb,amsmath}
\def\({\left(}
\def\){\right)}

\def\eb{\varepsilon}

\def\R {\mathbb{R}}

\newcommand{\be}{\begin{equation} }
\newcommand{\ee}{\end{equation} }

\def\ddt{\frac{d}{dt}}

\def \l {\langle}
\def \r {\rangle}
\def \p {\partial}
\def \px {\partial_{x}}

\def \and{\qquad\text{and}\qquad}

\def\Bbb{\mathbb}
\def\Dt{\partial_t}
\def\Dx{\Delta_x}
\def\({\left(}
\def\){\right)}

\def\eb{\varepsilon}

\def\eb{\varepsilon}

\def \px {\partial_{x}}

\def\R {\mathbb{R}}

\def \l {\langle}
\def \r {\rangle}

\def \p {\partial}

\def \and{\qquad\text{and}\qquad}

\def\Bbb{\mathbb}
\def\Dt{\partial_t}
\def\Dx{\Delta}
\catcode`\ç\active\def ç{{\c c}} \catcode`\Ç\active\def Ç{{\c C}}
\catcode`\ð\active\def ð{{\u g}} \catcode`\Ð\active\def Ð{{\u G}}
\catcode`\ý\active\def ý{{\i}} \catcode`\Ý\active\def Ý{{\. I}}
\catcode`\ö\active\def ö{{\" o}} \catcode`\Ö\active\def Ö{{\" O}}
\catcode`\þ\active\def þ{{\c s}} \catcode`\Þ\active\def Þ{{\c S}}
\catcode`\ü\active\def ü{{\" u}} \catcode`\Ü\active\def Ü{{\" U}}

\newtheorem{proposition}{Proposition}[section]
\newtheorem{theorem}[proposition]{Theorem}

\newtheorem{lemma}[proposition]{Lemma}
\theoremstyle{definition}

\newtheorem{remark}[proposition]{Remark}

\numberwithin{equation}{section}
\def\be{\begin{equation}}
\def\ee{\end{equation}}
\def\bp{\begin{proof}}
\def\ep{\end{proof}}

\def \no#1#2#3 {{\bf #1} (#3), #2.}
\def \eds#1#2#3 {#1, #2, #3.}
\title[Convective Cahn-Hilliard equation]
{Global solvability and blow up for the convective Cahn-Hilliard equations with concave potentials}
\author[]
{A. Eden, V.K. Kalantarov and S.V. Zelik}

\address{(A. Eden) Department of mathematics, Boðaziçi University,
 \newline\indent
Bebek, Istanbul, Turkey}

\address{(V.K.Kalantarov) Department of mathematics,
\newline\indent Ko{\c c} University, Rumelifeneri Yolu, Sariyer, Istanbul, Turkey
}
\address{(S.K.Zelik) Department of mathematics, \newline \indent University of Surrey
Guildford, GU2 7XH, UK}

 
\begin{document}

\begin{abstract}{We study initial boundary value problems for the  convective Cahn-Hilliard
equation $\Dt u +\px^4u +u\px u+\px^2(|u|^pu)=0$. It is well-known
that without the convective term, the solutions of this equation may
blow up in finite time for any $p>0$.
 In contrast to that, we show that the presence of the convective term $u\px u$
 in the Cahn-Hilliard equation prevents blow up at least
 for $0<p<\frac49$. We also show that the blowing up solutions still exist if $p$ is large enough ($p\ge2$).
 The related equations like Kolmogorov-Sivashinsky-Spiegel equation, sixth order convective Cahn-Hilliard equation, are also considered.}
\end{abstract}

\maketitle

% ----------------------------------------------------------------
\section{Intoduction}
It is well-known that the solutions of the semilinear heat equations with concave potentials

$$
\Dt u-\Dx u-u|u|^p=0
$$

blow up in finite time if $p>0$ and the initial energy is negative,
see e.g. \cite{AKS,KL,Le,LSU,SGKM} and references therein. However,
it is also established that the presence of the {\it convective}
terms in the semilinear parabolic equation prevents blow up if the
nonlinear source term is not growing very rapidly. For instance, the
solutions of the following convective heat equation
%$$
\begin{equation}\label{heat2}
\Dt u+u\partial_x u-\partial_x^2 u-u|u|^p=0
\end{equation}
%$$
in a bounded interval $\Omega=[-L,L]$ with Dirichlet boundary
conditions exist globally in time if $p\le1$ and the blowing up
solutions occur only if $p>1$, see \cite{CLS,Lenbc,LPSS,Str}, see
also \cite{Te} for the results on suppressing the blow up by adding
the sufficiently large linear convective terms in reaction-diffusion
equations.
\par
The main aim of the present paper is to study the analogous problems for  the following convective
Cahn-Hilliard (CH) equation with concave potentials
%$$
\begin{equation}\label{chintr}
\Dt u+\px^2(\px^2u+u|u|^p)+u\px u=0
\end{equation}
%$$
in a bounded segment $\Omega=[-L,L]$ endowed by {\it periodic} boundary conditions.
\par
The long-time behavior of solutions of initial boundary value
problems for  CH and related equations are intensively studied by
many authors, see \cite{BBT,BB,ES,Ka,GW,Po,Tem} and references
therein. For instance, the existence of the blowing up solutions for
equation \eqref{chintr} {\it without} the convective term $u\px u$
is known for any $p>0$, see e.g. \cite{ES,EGW,Nov,Pu}. However,
based on the reaction-diffusion experience mentioned above, one may
expect that the convective term prevents blow-up here as well. We
will show below that this is indeed the case. Namely, the following
theorem can be considered as the main result of the paper.

\begin{theorem}\label{main} Let $0<p<\frac49$. Then, for every initial data $u_0\in L^2(\Omega)$ with zero mean,
problem \eqref{chintr} possesses a unique solution $u(t)$ which exists for all $t\ge0$ and remains bounded when $t\to\infty$.
\par
Let $p\ge2$. Then, there are initial data $u_0\in L^2(\Omega)$ with zero mean such that the corresponding solution blows up in finite time.
\end{theorem}
Note that, in contrast to the situation with the semilinear heat
equations, the result of Theorem \ref{main} is not complete in the
sense that we know nothing about the behavior of solutions for
$\frac49\le p<2$. Indeed, the sharp result for the semilinear heat
equations stated above is strongly based on the maximum principle
which we do not have for the CH equations, so we have to use
alternative less powerful methods. We hope to return to this problem
somewhere else.
\par
The paper is organized as follows.
\par
In Section \ref{s1}, we introduce the notations and main technical tools which are necessary for our proof of Theorem \ref{main}.
\par
Section \ref{s2} is devoted to the proof of the blow up for the case $p\ge2$. Actually, in the non-convective case,
sufficient conditions of blow up of solutions for the CH equation can be  established by concavity method of Levine (see \cite{Le})
rewriting the equation in the form  of
a nonlinear differential operator equations  of the form
\begin{equation}\label{DO}
Pu_t + Au =F(u)
\end{equation}
 in a Hilbert space (see \cite{Ka}). Here $P,A$ are positive self-adjoint operators and $F(\cdot)$ is a nonlinear gradient operator .
 But the convective CH equation \eqref{chintr} can not be written in the form \eqref{DO}, so the method does not work directly and
 its adaptation to our situation  requires some more delicate arguments and works only under the extra assumption $p\ge2$.
\par
In Section \ref{s3}, we prove the global existence and dissipativity of solutions of \eqref{chintr} in the case $0<p<\frac49$.
Our proof uses the so-called Goodmann trick (see \cite{G}) which is now-a-days the standard
(for the theory of Kuramoto-Sivashinki equation)
 method to "extract" the dissipation from the convective term. However, in order to compensate the concave term $u|u|^p$,
 we need to chose the auxiliary function in this method depending on the size of the initial data and then use the so-called
 Gronwall lemma with parameter (see \cite{GPZ} and also \cite{Z}, where this lemma was implicitly used to establish the global
 existence of solutions for the Navier-Stokes equations in a strip).
 Note that our approach also essentially employs the improved construction of the auxiliary function given in \cite{BG}.
\par
Finally, in Section \ref{s4}, we consider some related problems
which can be straightforwardly solved using the methods developed in
the paper. In particular, we study here the problem of obtaining
uniform in $\delta\to0$ upper bounds for the attractor of the
so-called Kolmogorov-Sivashinsky-Spiegel equation
%$$
\begin{equation}\label{kss1}
\Dt u +\px^4u+\px^2(2u-\delta u^{3})+u\px u=0.
\end{equation}

Furthermore, the so-called  sixth order convective Cahn-Hilliard equation
%$$
\be\label{sx1intr} \Dt
u -\px^4(\px^{2}u+u-u^3)+ u\px u=0.
 \ee

%as well as viscous convective CH equation
%$$
%\be\label{vcchintr}
% \Dt ( u -\nu
%\px^2u) +\px^4u+\px^{2}\(\lambda u -u^3\)+ \delta u\px u=0
%\ee
%$$
are considered there.

\section{Notations and preliminaries}\label{s1}
In this section, we introduce the notations which will be used
throughout the paper and introduce some technical tools important
for what follows.
\par
As usual, we denote by $H^m=H^m(-L,L)$  the Sobolev space of distributions whose derivatives up to order $m$ belong to $L^2(-L,L)$.
We write $H=H(-L,L)$ instead of $H^0=L^2(-L,L)$ and $(\cdot,\cdot)$ stands for the usual scalar product in the Hilbert space $H$.
 \par
 The closure of $C^\infty_0(-L,L)$ in $H^m(-L,L)$ will be denoted by $H^m_0=H^m_0(-L,L)$ and $H^m_{per}=H^m_{per}(-L,L)$ stands
 for the subspace of $H^m$ which consists of $2L$ periodic functions. This definition works only for $m\in\Bbb N$,
  for negative or/and fractional $m$s, the corresponding Sobolev spaces are defined in a standard way using the duality
  and interpolation arguments respectively, see e.g. \cite{Tem}.
\par
For every $u\in H^m_{per}$, $m\in\Bbb N$, we introduce the mean
value operator
$$
\l u\r:=\frac1{2L}\int_{-L}^Lu(x)\,dx
$$
and denote by $\dot H^m_{per}=\dot H^m_{per}(-L,L)$ the subspace of $H^m_{per}$ which consists of functions with zero mean:
$$
\dot H^m_{per}:=\{u\in H^m_{per},\ \l u\r=0\}.
$$
We also introduce the inverse of the Laplace operator
$P:=(-\px^2)^{-1}$ defined on the functions with zero mean. It is
well known that this operator gives the isomorphism between the
spaces $\dot H^m_{per}$ and $\dot H^{m+2}_{per}$ for every $m\in\Bbb
N$:
$$
P: \dot H^m_{per}\to\dot H^{m+2}_{per},\ \ P\dot H^m_{per}=\dot H^{m+2}_{per}
$$
and the following relations hold for all $u\in \dot H^1$:

 \be\label{Po1} \|u
\|_H\leq \bar d_0\|\px u\|_H, \ \ \|P^{\frac12}(\px u)\|_H = \|u\|_H, \
\bar d_0:=\frac L\pi. \ee
%$$
Applying the operator $P$ to the both sides of equation
\eqref{chintr}, we rewrite it in the equivalent, but more convenient
(for our purposes) form
%$$
\begin{equation}\label{cch-not}
P\Dt u-\px^2 u-P(u\px u)=u|u|^p-\l u|u|^p\r.
\end{equation}
%$$
We say that a function $u(t,x)$ is a weak solution of \eqref{chintr} (or equivalently of \eqref{cch-not})
 on the time interval $t\in[0,T]$ if
%$$
\begin{equation}\label{smo}
u\in C([0,T],\dot H_{per})\cap L^2([0,T],\dot H^2_{per}),\ \ u\in L^{p+1}([0,T],L^{p+1}(-L,L))
\end{equation}
%$$
and equation \eqref{cch-not} is satisfied in the sense of distributions.
\par
The following theorem gives the {\it local} well-posedness of this problem.
\begin{theorem}\label{Th.local} Let $0<p<4$. Then, for every $u_0\in\dot H_{per}$, the problem \eqref{cch-not} possesses a unique
weak solution defined on the time interval $t\in[0,T]$, where $T=T(u_0)>0$ depends on the $H$-norm of the initial data $u_0$
\end{theorem}
The proof of this theorem is standard and, by this reason, is omitted.
\begin{remark} As follows from the embedding theorem,
$$
u\in C([0,T],\dot H_{per})\cap L^2([0,T],\dot H^2_{per})
$$
implies that $u\in L^{10}([0,T],L^{10})$. Therefore, the assumption
$p<4$ guarantees that the nonlinear term  $u|u|^p$ belongs to
$L^2([0,T],H)$ and therefore it is subordinated to the linear terms
in the equation.
\par
Let us mention also that the usual parabolic smoothing property
works for such weak local solutions, so the factual smoothness of
the solution $u(t)$ for $t>0$ is restricted only by the smoothness
of the non-linearity $u|u|^p$ at $u=0$. In particular, since this
nonlinearity is of at least $C^{1+\alpha}$ for some $\alpha>0$
(depending on $p$), one can show that we have at least $u(t)\in
C^3(-L,L)$ for $t\in(0,T]$ and any weak solution constructed in
Theorem \ref{Th.local}. Since this regularity is more than enough to
justify all  estimates used in the paper, we will not return to the
questions of {\it local} well-posedness in what follows and will
only concentrate ourselves on derivation of the a priori estimates
which guarantee the {\it global} well-posedness or finite-time blow
up of solutions.
\par
In the case $p\ge4$, one can obtain the similar local well-posedness
result just using more regular solutions, say,
$$
u\in C([0,T],\dot H^1_{per})\cap L^2([0,T],\dot H^3_{per})
$$
and starting from more regular initial data $u_0\in\dot H^1_{per}$.
\end{remark}
 We conclude this section by stating  two crucial lemmas, one of them will allow us to
  find sufficient conditions for blow up of solutions of the CH equation \eqref{chintr} and the other
  one will give the part of Theorem \ref{main} related with the global solvability.\\
\begin{lemma}\label{Lev1} (\cite{Le})
Let  $\Psi(t)$ be twice continuously differentiable function that
satisfies the inequality \be\label{01} \Psi''(t)
\Psi(t)-(1+\alpha)\left[\Psi(t)\right]^2\geq 0, \ t>0, \ee and \be\label{001}
\Psi(0)>0, \Psi'(0)>0, \ee where $\alpha
>0$ is a given number. Then there exists $$t_1 \leq T_1=\frac{\Psi(0)}{\alpha \Psi'(0)}$$ such that
$$
\Psi(t)\rightarrow \infty \ \ \mbox{as} \ \  t\rightarrow t_1^{-}.
$$
\end{lemma}

\begin{lemma}\label{GrZel}(\cite{GPZ,Pata}) Suppose that $\alpha >\beta\geq 1$ and
 $\gamma\geq 0$ are given numbers that  satisfy the inequality,
%$$
\begin{equation}\label{1.cond}
\frac{\beta -1}{\alpha-1}<\frac1{\gamma+1},
\end{equation}
%$$
Suppose also that $\Psi$ is a non-negative absolutely continuous
function on $[0,\infty)$ which satisfies, for some numbers $K\geq 0,
\eb_0 >0$, $M>0$ and for every $\eb \in (0,\eb_0)$  the differential
inequality
%$$
\begin{equation}\label{Gron-par}
\Psi'(t)+\eb \Psi(t)\leq K\eb^\alpha\left[\Psi(t)\right]^\beta
+M\eb^{-\gamma}
\end{equation}
%$$
Then there exist a monotone function $Q:\R_+\to\R_+$ and a positive
number $\kappa>0$ such that
%$$
\begin{equation}\label{Gron-dis}
\Psi(t)\leq Q(\Psi(0))e^{-\kappa t}+Q(M).
\end{equation}
%$$
Moreover, the dissipative estimate \eqref{Gron-dis} remains true if the function $\Psi\ge0$ is only continuous and
satisfies the integrated version of inequality \eqref{Gron-par}
%$$
\begin{equation}\label{Gron-par1}
\Psi(t)\le \Psi(0)e^{-\eb t}+K\eb^\alpha\int_0^te^{-\eb(t-s)}[\Psi(s)]^\beta\,ds+M\eb^{-\gamma-1}
\end{equation}
%$$
for all $t\ge0$.
\end{lemma}

\section{Blow up of solutions to convective CH
equations}\label{s2}

In this section we consider the following problem in $\Omega:=[-L,L]$:
$$
\begin{cases}\Dt u+\px^2(\px^2 u +u^3)+u\px u=0,\\
u\big|_{t=0}=u_0.
\end{cases}\eqno(A)
$$
endowed by periodic boundary conditions. According to Theorem \ref{Th.local}, this problem has
a unique local weak solution for all $u_0\in \dot H_{per}$ and it is equivalent to the following one:
$$
\begin{cases}
P\Dt u-\px^2 u+P(u\px u)=u^3-\l u^3\r,\\
u\big|_{t=0}=u_0,
\end{cases}
\eqno(A_1)
$$
see Section \ref{s2}.
\par
 Our aim is to show that for some
class of initial functions the corresponding solutions blow up in a
finite time. For simplicity, we restrict ourselves to consider only
the case $p=3$ in equation \eqref{chintr}, although as it is not
difficult to see, the similar arguments work for all $p\ge3$.
\par
 The main result of the section is the following theorem.
\begin{theorem}\label{gn} Suppose that $u$ is a solution of the problem $(A_1)$ corresponding to the initial
data $u_0\in\dot  H_{per}$ which is not equal zero identically and
\begin{equation}
E_0:=-\frac{\lambda}{2}\|P^{1/2} u_0\|^2-\frac12\|\px
u_0\|^2+\frac14(u_0^4,1)\geq 0,
\end{equation}
where $\lambda=\frac12+3d_0^2$ and
$d_0:=\max\{\bar d_0,1\}$, see   \eqref{Po1}.
\par
Then there exists $t_1<\infty$ such that
$$
\| P^{1/2}u(t)\|\rightarrow \infty, \ \ {as} \ \ t\rightarrow t_1^-.
$$
\end{theorem}

\begin{proof}
We make the change $u=e^{\lambda t} v$. Then the function $v$ solves the problem
\be\label{cc1}
\begin{cases}
P\Dt v+\lambda Pv-\px^2 v+e^{\lambda t}P(v\px v)=e^{2\lambda t}(v^3-\l v^3\r),\\
v\big|_{t=0}=u_0.
\end{cases}
 \ee
Let us consider the function \be\label{Psi}
\Psi(t):=\int_0^t\|P^{1/2}v(\tau)\|^2d\tau +C_0, \ee
 where $C_0$ is
some positive parameter to be chosen below. Clearly \be\label{cc2}
\Psi'(t)=\|P^{1/2}v(t)\|^2=2\int_0^t(P\p_{\tau}v,v)d\tau
+\|P^{1/2}u_0\|^2. \ee Employing the equation \eqref{cc1} we also
obtain that \be\label{cc3} \Psi''(t)=2(P\Dt v,v)=-2\lambda
\|P^{1/2}v\|^2-2\|\px v\|^2-2e^{\lambda t}(P(v\px v),v)+2e^{2\lambda
t}(v^4,1). \ee

By using the Cauchy inequality with $\eb$ and the Schwarz inequality
we obtain
\begin{multline*}
2e^{\lambda t}|(P(v\px v),v)|=2e^{\lambda t}|(P^{1/2}(v\px
v),P^{1/2}v)|\leq 2e^{\lambda t}\|P^{1/2}(v\px v)\|\|P^{1/2}v\|\leq\\
\leq 2d_0e^{\lambda t}\|v^2\|\|P^{1/2}v\|\leq \eb_1e^{2\lambda
t}(v^4,1)+\frac{d_0^2}{\eb_1}\|P^{1/2}v\|^2.
\end{multline*}

Thus \eqref{cc3} implies \be\label{cc4} \Psi''(t)\geq
-\left(2\lambda+\frac{d_0^2}{\eb_1}\right) \|P^{1/2}v\|^2-2\|\px
v\|^2+(2-\eb_1)e^{2\lambda t}(v^4,1). \ee

Multiplying the equation \eqref{cc1} by $\Dt v$ and integrating over
$(-L,L)$ we obtain the second main energy equality :

\be\label{cc5} \ddt E(t)= \|P^{1/2}\Dt v\|^2+\frac{\lambda}2
e^{2\lambda t}(v^4,1)+e^{\lambda t}(P(v\px v),\Dt v), \ee where
$$
E(t):=-\frac{\lambda}{2}\|P^{1/2} v\|^2-\frac12\|\px
v\|^2+\frac14e^{2\lambda t}(v^4,1).
$$

Let us estimate the last term on the right hand side of \eqref{cc5}
: \be\label{cc6} e^{\lambda t}|(P(v\px v),\Dt v)|\leq e^{\lambda
t}d_0\|v^2\|\|P^{1/2}\Dt v\|\leq \eb_2\|P^{1/2}\Dt
v\|^2+\frac{d_0^2}{4\eb_2}e^{2\lambda t}(v^4,1). \ee

If  \be\label{e2}\lambda \geq \frac{d_0^2}{2\eb_2}\ee
 then it
follows from \eqref{cc5} and \eqref{cc6} that
$$
\ddt E(t)\geq (1-\eb_2)\|P^{1/2}\Dt v\|^2.
$$

Hence

 \be\label{cc7} E(t)\geq
(1-\eb_2)\int_0^t\|P^{1/2}\partial_{\tau} v(\tau)\|^2d \tau+E_0. \ee
It follows from \eqref{cc4} that
$$
\Psi''(t)\geq4(2-\eb_1)E(t)+\left[2(1-\eb_1)\lambda-\frac{d_0^2}{\eb_1}\right]\|P^{1/2}v\|^2+(2-2\eb_1)\|\px
v\|^2.
$$
Let us take $\eb_1=\frac12$ (it is important to have
$4(2-\eb_1)>4$). Then we have the following estimate for
$\Psi''(t)$:
 \be\label{cc8}
\Psi''(t)\geq 6E(t)+(\lambda -2d_0^2)\Psi'(t)+\|\px v\|^2. \ee

Due to \eqref{cc7} and \eqref{cc8} with $\eb_2=\frac16$ (note that
for this $\eb_2$ the inequality \eqref{e2} is satisfied) we have:
\be\label{cc9} \Psi''(t)\geq 5\int_0^t\|P^{1/2}\partial_{\tau}
v(\tau)\|^2d \tau+(\lambda -2d_0^2)\Psi'(t). \ee Here we have used
the condition $E(0)=E_0\geq 0.$
 By using
\eqref{cc2}and \eqref{cc9} we obtain the following inequality
\begin{multline}\label{cc10}
\Psi''(t)\Psi(t)-\frac54\left(\Psi'(t)\right)^2\geq
5\left(\int_0^t\|P^{1/2}\partial_{\tau} v\|^2d \tau
\right)\left(\int_0^t\|P^{1/2}
v\|^2d \tau+C_0\right)+\\
-5\left(\int_0^t(P^{1/2}\partial_{\tau}v,P^{1/2}v)d\tau
+\frac12\|P^{1/2}u_0\|^2\right)^2+(\lambda -2d_0^2)\Psi\Psi'=\\
5\left[\left(\int_0^t\|P^{1/2}\partial_{\tau} v\|^2d \tau+C_0
\right)\Psi(t)-\left(\int_0^t(P^{1/2}\partial_{\tau}v,P^{1/2}v)d\tau
+\frac{\|P^{1/2}u_0\|^2}2\right)^2\right]+\\
(\lambda -2d_0^2)\Psi(t)\Psi'(t)-C_0\Psi(t).
\end{multline}
Let us choose in \eqref{cc10}
$$
C_0=\frac12\|P^{1/2}u_0\|^2.
$$
Then the  expression in the square brackets is nonnegative due to
Cauchy-Schwatz inequality. Therefore we have \be\label{cc9a}
\Psi''(t)\Psi(t)-\frac54\left(\Psi'(t)\right)^2\geq (\lambda
-2d_0^2)\Psi(t)\Psi'(t)-C_0\Psi(t). \ee According to \eqref{cc8} the
function $\Psi'(t)$ is a non-decreasing function. Thus
$$\Psi'(t)=\|P^{1/2}v(t)\|^2\geq \|P^{1/2}u_0\|^2.$$
Hence we obtain
from \eqref{cc9a}
$$
\Psi''(t)\Psi(t)-\frac54\left[\Psi'(t)\right]^2\geq (\lambda
-2d_0^2)\Psi(t)\|P^{1/2}u_0\|^2-C_0\Psi(t)\geq
[2(\lambda-2d_0^2)-1]C_0\Psi(t).
$$
Finally noting that $\lambda =\frac12+3d_0^2$ we obtain:
$$
\Psi''(t)\Psi(t)-\frac54\left[\Psi'(t)\right]^2\geq 0.
$$
Hence due to the Lemma \ref{Lev1} the
statement of the Theorem \ref{gn} holds true.
\end{proof}
\begin{remark} It is clear that Theorem \ref{gn} is true for
solutions of the equation (A) under the homogeneous Dirichlet's
boundary conditions.

\end{remark}
\begin{remark} It is easy to see that the result of the Theorem \ref{gn} remains true also for the multi-dimensional convective CH equation of the form
$$
\Dt u+\Delta(\Delta u +u^3)+u\vec{b}.\nabla u=0,
$$
where $\vec b\in L^\infty(\Omega,\R)$ is a given vector field and $u$ satisfies  the appropriate boundary conditions.

\end{remark}
\section{The convective CH equations: global existence}\label{s3}
In this section, we continue our study  of the convective CH
equation:
%$$
\begin{equation}\label{1.cch}
\Dt u+u\px u+\px^2(\px^2u+u|u|^p)=0
\end{equation}
%$$
on the interval $\Omega=[-1,1]$ (for simplicity, we take $L=1$ here) endowed with the periodic
conditions. We have seen in the previous section that this equation
possesses the blowing up in finite time solutions if $p\ge3$. The
aim of the present section is to show that the presence of the
convective term {\it prevents} the blow up if the exponent $p$ is
not large.

Namely, the following theorem is the main result of the section.
\begin{theorem}\label{Th.global} Let the exponent $0\le p<\frac49$. Then, for every $u_0\in L^2([-1,1])$ with zero mean, problem \eqref{1.cch}
possesses a unique solution defined for all $t\ge0$ and the following estimate holds:
%$$
\begin{equation}\label{CH-dis}
\|u(t)\|_{L^2}\le Q(\|u_0\|_{L^2})e^{-\alpha t}+C_*,
\end{equation}
%$$
where the positive constants $C_*$ and $\alpha$ and the monotone increasing function $Q$ are independent of $t$ and $u_0$.
\end{theorem}
\begin{proof} Since the local existence and uniqueness theorem for the equation \eqref{1.cch} equation is standard and immediate,
we only need to verify the dissipative estimate \eqref{CH-dis}.
As usual, we start with the case of odd periodic
solutions and use the following Lemma which is in fact proved in \cite{BG}
\begin{lemma}\label{Bronski} For every sufficiently large $N$ there exists a 2-periodic function $\phi$ with zero mean such that
%$$
\begin{equation}\label{1.phi}
\|\phi\|_{H^2}\le CN^{3/2},\ \ \|\phi\|_{L^\infty}\le CN
\end{equation}
%$$
with constant $C$ independent of $N$, such that, for every $u\in
H^2(\Omega)$ with $u(0)=0$,
%$$
\begin{equation}\label{1.estt}
\|u_{xx}\|^2_{L^2}-(\phi_x,|u|^2)\ge N\|u\|^2_{L^2}.
\end{equation}
%$$
\end{lemma}
The function $\phi$ with the desired properties  is constructed up to
 scaling in \cite{BG}. Indeed, it is proved there that, for any sufficiently large $L$,
  there exists a $2L$-periodic function $\psi\in H^2(-L,L)$ such that
$$
\|\psi\|_{H^2}\le CL^{3/2},\ \ \|\psi\|_{L^\infty}\le CL
$$
and, for any $u\in H^2(-L,L)$ with $u(0)=0$, the following
inequality is satisfied:
$$
\int_{-L}^L\left([\px^2u(x)]^2-u^2(x) \px \psi \right),dx\ge
\frac12\int_{-L}^Lu^2(x) dx.
$$
Scaling $x=Ly$, $\phi=L^3\psi$ and $N=CL^4$, we end up with
\eqref{1.estt}.
\par
We now return to  the key a priori estimate \eqref{CH-dis} for the {\it
odd} periodic solutions of \eqref{1.cch}. To this end, for any large
$N$, we multiply equation \eqref{1.cch} by $v:=u-\phi$ where
$\phi=\phi_N$ is constructed in Lemma \ref{Bronski}. Then, after
some transformations, we get
%$$
\begin{equation}\label{1.1}
\Dt \|v\|^2_{L^2}+2\|\px^2u\|^2_{L^2}-(\px\phi,|u|^2)=
-2(u|u|^p,\px^2u-\px^2\phi)+2(\px^2u,\px^2\phi).
\end{equation}
%$$
Using the Cauchy-Schwartz inequality together with \eqref{1.phi} and
\eqref{1.estt}, we get
%$$
\begin{equation}\label{1.2}
\Dt \|v\|^2_{L^2}+1/2\|\px^2u\|^2_{L^2}+N\|u\|^2_{L^2}\le
\|u\|^{2(p+1)}_{L^{2(p+1)}}+CN^3.
\end{equation}
%$$
Note also that, due to Lemma \ref{Bronski}, we see that, for every $q\in[1,\infty]$,
%$$
\begin{equation}\label{two-side}
\|u\|_{L^q}-CN\le \|v\|_{L^q}\le \|u\|_{L^q}+CN,
\end{equation}
%$$
where the constant $C$ is independent of $N$.

%Using \eqref{1.phi} again together with the fact that $p\le1/2$
%(and, therefore, $2(p+1)\le3$), we can replace $u$ by $v$ in
%\eqref{1.1} and arrive at
%$$
%\begin{equation}\label{1.2}
%\Dt \|v\|^2_{L^2}+1/2\|\px^2v\|^2_{L^2}+N\|v\|^2_{L^2}\le
%\|v\|^{2(p+1)}_{L^{2(p+1)}}+CN^3.
%\end{equation}
%$$
Applying the interpolation inequality
$$
\|u\|_{L^{2(p+1)}}^{2(p+1)}\le
C\|u\|_{L^2}^{\frac{3p+4}2}\|\px^2u\|^{\frac p2}_{L^2}\le
\frac12\|\px^2u\|^2_{L^2}+C\|u\|_{L^2}^{2\frac{3p+4}{4-p}}
$$
to estimate  the right-hand side of \eqref{1.2},
we end up with
%$$
\begin{equation}\label{1.3}
\Dt \|v\|^2_{L^2}+N\|u\|^2_{L^2}\le
C\|u\|_{L^2}^{2\frac{3p+4}{4-p}}+CN^3.
\end{equation}
%$$
In order to derive the desired dissipative estimate for $u$ from
\eqref{1.3}, we'll use the Lemma \ref{GrZel}.
To this end, we scale time
$t=N^2\tau$ and introduce $\eb=1/N$. Then,
\eqref{1.3} reads
$$
\partial_\tau\|v\|^2_{L^2}+\eb\|u\|^2_{L^2}\le C\eb^2\|u\|_{L^2}^{2\frac{3p+4}{4-p}}+C\eb^{-1}.
$$
Integrating this inequality in time, introducing $\Psi(t):=\|u(t)\|^2_{L^2}$ and using \eqref{two-side}, we
end up with
%$$
\begin{equation}\label{ggoodd}
\Psi(t)\le \Psi(0)e^{-\eb t}+C\eb^2\int_0^te^{-\eb(t-s)}[\Psi(s)]^{\frac{3p+4}{4-p}}\,ds+C\eb^{-2},
\end{equation}
%$$
where $C$ is independent of $\eb\to0$. Equation \eqref{ggoodd} ha the form of \eqref{Gron-par1} with
$\alpha=2$, $\beta=\frac{3p+4}{4-p}$ and $\gamma=1$ and the
condition \eqref{1.cond} reads
$$
\frac {4p}{4-p}<\frac12.
$$
This condition is i  satisfied if and only if $p<\frac49$. Thus, due to Lemma \ref{GrZel},
the {\it odd} solutions of \eqref{1.cch} cannot blow up in finite
time and satisfy the dissipative estimate \eqref{CH-dis} if $p<\frac49$, so in the particular case of odd initial data, Theorem \ref{Th.global} is proved.
\par
We are now ready to consider the general case of periodic $u\in \dot H_{per}$ with using the so-called Goodman trick, see \cite{G}.
 Namely, we consider a {\it circle} of shifted functions $\phi_s(x):=\phi(x+s)$, $s\in\R$ and introduce
$$
R(t):=\min_{s\in\R}\|u(t)-\phi_s\|^2.
$$
Then, due to \eqref{1.phi}, we have the analogue of \eqref{two-side}:
%$$
\begin{equation}\label{two-sided1}
\|u(t)\|^2_{L^2}-CN\le R(t)\le \|u(t)\|_{L^2}^2+CN
\end{equation}
%$$
and for the minimizer $\phi_{s(t)}$, we have
%$$
\begin{equation}\label{1.g}
(u(t),\partial_x\phi_{s(t)})\equiv0,
\end{equation}
%$$
and, at least formally (see \cite{G} for the justification),
%$$
\begin{multline*}
\frac12\frac d{dt}R(t)=\frac12\frac
d{dt}\|u(t)-\phi_{s(t)}\|^2_{L^2}=(\Dt
u(t),u(t)-\phi_{s(t)})-\\-s'(t)(\partial_x\phi_{s(t)},u(t)-\phi_{s(t)})=
  (\Dt u(t),u(t)-\phi_{s(t)}).
\end{multline*}
%$$
Thus, multiplying equation \eqref{1.cch} by $2(u(t)-\phi_{s(t)})$
and arguing as before, we get
%$$
\begin{equation}\label{1.good}
\frac d{dt}R(t)+3/2\|\px^2u\|^2_{L^2}+(\px\phi_{s(t)},|u|^2)\le
\|u\|^{2(p+1)}_{L^{2(p+1)}}+CN^3
\end{equation}
%$$
However, in contrast to the odd case, we cannot apply directly Lemma
\ref{Bronski} since the condition $u(t,s(t))=0$ is
 not necessarily satisfied. So, we need to introduce a time dependent "constant"  $c(t):=u(t,s(t))$
 and a function $w(t)=u(t)-c(t)$ for which the conditions of the Lemma \ref{Bronski} are satisfied, and we have
 $$
 \|\px^2w\|^2_{L^2}+(\partial_{x}\phi_{s(t)},|w|^2)\ge N\|w\|^2.
 $$
 Note that, due to the zero mean condition on $u$,
 $$
 \|w\|^2=\|u\|^2_{L^2}+|\Omega|c(t)^2\ge \|u\|^2
 $$
 and, due to the orthogonality condition \eqref{1.g},
 $$
 (\partial_x\phi_{s(t)},|w|^2)=(\partial_x\phi_{s(t)},|u|^2)-2c(t)(\partial_x\phi_{s(t)},u(t))+
 c(t)^2(\partial_x\phi_{s(t)},1)=(\partial_x\phi_{s(t)},|u|^2).
 $$
 Therefore,
 $$
 \|\px^2u\|^2_{L^2}+(\partial_{x}\phi_{s(t)},|u|^2)=
 \|\px^2w\|^2_{L^2}+(\partial_{x}\phi_{s(t)},|w|^2)\ge N\|u\|^2
 $$
 and \eqref{1.good} implies that
 %$$
\begin{equation}\label{1.ggood}
\frac d{dt}R(t) + N\|u\|^2_{L^2}\le
C\|u\|_{L^2}^{2\frac{3p+4}{4-p}}+CN^3.
\end{equation}
%$$
Finally, integrating \eqref{1.ggood} in time, using \eqref{two-sided1} and arguing
 exactly as in the case of odd initial data, we derive
the desired dissipative estimate \eqref{CH-dis} for general $u_0\in\dot H_{per}$ and finish the proof of the theorem.
\end{proof}

\section{Related problems}\label{s4}
In this section, we apply the above considered methods to some equations which are, in a sense, close to the Kuramoto-Sivashinski
and Cahn-Hilliard equations, such as Kolmogorov-Sivashinski-Spiegel equations and
for the sixth order convective CH equations.
\subsection{Convective CH Equation vs
Kolmogorov-Sivashinsky-Spiegel equation}

We  consider now the problem \be\label{1cch}
\begin{cases}
\Dt u +\px^4u+\px^2\(2u-\delta u^3\)+u\px u=0,\\
u\big|_{t=0}=u_0
\end{cases}
\ee
in the domain $\Omega:=(-L,L)$ endowed by the periodic boundary conditions.
We are going to  show that the estimate obtained in \cite{FNT}
for the size of the absorbing ball can be improved at least for
small values of $\delta $ where one expects that the
Kuramato-Sivashinsky dynamics will dominate. Indeed, to the best of our knowledge, all previous methods of obtaining the dissipative estimates for this equation utilize only the dissipativity which comes from the cubic term ignoring the extra dissipation provided by the convective term. As a result, the obtained estimates were {\it divergent} as $\delta\to0$, see \cite{FNT,CFNT,EFNT,EK}. In particular, the radius of the absorbing ball in $\dot H_{per}$ constructed in \cite{CFNT} behaves like $L^3\delta^{-1/2}$ as $\delta\to0$.
\par
Using below the technique related to the Kuramoto-Sivashinski equation (analogous to what is used in Section \ref{s3}, we show that the radius of the absorbing ball remains bounded as $\delta\to0$. Since we are not interested in the dependence of this radius on $L$, we set $L=1$ for simplicity. Then, the following theorem holds.
\begin{theorem}\label{Th.something} Let $L=1$. Then, for every $u_0\in\dot H_{per}$, problem \eqref{1.cch} possess a unique solution and the following estimate holds:
%$$
\begin{equation}\label{l2}
\|u(t)\|_{L^2}^2\le C\|u_0\|^2e^{-\alpha t}+C_*
\end{equation}
%$$
where the positive constants $C$, $\alpha$ and $C_*$ are independent of $\delta\to0$.
\end{theorem}
\begin{proof} The existence and uniqueness for that equation is well-known, see e.g., \cite{EK}, so we only derive the uniform estimate \eqref{l2}. Analogously to Section \ref{s3}, we start with the case of odd initial data $u_0$ and multiply \eqref{1cch} by $v=u-\varphi$ where $\varphi=\varphi_N$ is the same as in Lemma \ref{Bronski} and $N$ will be fixed below. Then, after the obvious transformations, we have
%$$
\begin{multline}\label{1.1.1}
\Dt \|v\|^2_{L^2}+2\|\px^2u\|^2_{L^2}-(\px\phi,|u|^2)+6\delta(u^2,|\partial_x u|^2)=\\=
6\delta(u^2\px u,\px\phi)+4\|\px u\|^2_{L^2}-4(\px u,\px\varphi)+2(\px^2u,\px^2\phi).
\end{multline}
%$$
We estimate the first term in the right-hand side of \eqref{1.1.1} via the Cauchy-Schwartz inequality:
%$$
\begin{multline}
6\delta(u^2\px u,\px\phi)=6\delta(u\px u,u\px\phi)\le 6\delta(u^2,|\px u|^2)+6\delta(|\px \phi|^2,u^2)\le\\\le 6\delta(u^2,|\px u|^2)+CN^3\delta\|u\|^2_{L^2},
\end{multline}
%$$
where we have used that $\|\px\phi\|_{L^\infty}\le C\|\phi\|_{H^2}\le CN^{3/2}$. The rest terms in the right-hand side can be estimated in a standard way using the interpolation inequality $\|w\|_{H^1}\le C\|w\|_{H^2}^{1/2}\|w\|_{L^2}^{1/2}$ and the Cauchy-Schwartz inequality. Therefore, using again \eqref{1.phi}, we see that for sufficiently large $N$,
%$$
\begin{equation}\label{stupid}
\Dt \|v\|^2_{L^2}+(\beta N-\alpha(1+\delta N^3))\|u\|^2_{L^2}\le
CN^3.
\end{equation}
%$$
where the positive constants $\beta$, $\alpha$ and $C$ are independent of $\delta$ and $N$. Finally, fixing $N$ such that $\beta N\ge 3\alpha$, we see that, for $\delta$ being small enough that $\delta N^3\le1$, the following inequality holds:
$$
\Dt \|v\|^2_{L^2}+\alpha\|u\|^2_{L^2}\le CN^3
$$
and the Gronwall inequality applied to it gives the desired uniform estimate \eqref{l2}. Thus, the theorem is proved in the particular case of odd initial data. The general case can be reduced to
that particular one using the Goodman trick again, exactly as in Section \ref{s3}. So, the theorem is proved.
\end{proof}
\begin{remark} It looks natural to consider the mixture of problems \eqref{1.cch} and \eqref{1cch}, namely,
$$
\Dt u+uu_x+(u_{xx}+u|u|^p-\delta u^3)_{xx}=0
$$
with $p<\frac49$ and study the limit $\delta\to0$. However, the above method does not work at least directly
in this case (at least without stronger assumptions on $p$). Indeed, as we see from \eqref{stupid}, if we allow $N$
to be dependent on the norm of the initial data (as in Section \ref{s3}), we also need to decrease $\delta$ in
the dependence on the initial data. Thus, the problem of obtaining the uniform with respect to $\delta\to0$ estimates for that case remains open.
\end{remark}

\subsection{Sixth order convective  Cahn - Hilliard equation}

 We consider here the following sixth order convective Cahn-Hilliard equation
%$$
\be\label{sx1}\begin{cases} \Dt
u -\px^4(\px^{2}u+u-u^3)+ u\px u=0,\\
u\big|_{t=0}=u_0
\end{cases}
 \ee
 %$$
on a bounded interval $\Omega=(-L,L)$.
 This equation was
derived in \cite{SGND} as a model of process of growing crystalline
surface with small slopes that undergoes faceting
where $u=\px h$ is the slope of a surface $h(x,t)$. In \cite{KR} the
authors proved existence and uniqueness of the global solution to
initial boundary value problem for \eqref{sx1} under periodic
boundary conditions.\\
We are going to show that the semigroup generated by the initial
boundary value problem for \eqref{sx1} under the {\it periodic} boundary conditions is dissipative (the case of, say,
Dirichlet boundary conditions can be treated similarly), namely, the following proposition holds.
\begin{proposition}\label{abb1} For any $u_0\in\dot H^{-1}_{per}$, problem
\eqref{sx1} possesses a unique solution $u\in C(\R_+,\dot H^{-1}_{per})\cap L^2_{loc}(\R_+,\dot H^2_{per})$
and the following dissipative estimate holds:
%$$
\begin{equation}\label{disdis}
\|u(t)\|_{\dot H^{-1}_{per}}\le Q(\|u_0\|_{\dot H^{-1}_{per}})e^{-\alpha t}+C_*,
\end{equation}
%$$
for the positive constants  $\alpha$ and $C_*$ and monotone function $Q$ which are independent of $t$ and $u_0$.
\end{proposition}
\begin{proof}
Applying to both sides of \eqref{sx1} the operator $P=(-\px^2)^{-1}$ we get an
equivalent problem \be\label{sxE}
\begin{cases}
P\Dt u +\px^2(\px^2 u+u-u^3)-P(u\px u)=0,\\
u\big|_{t=0}=u_0.
\end{cases}
\ee
Multiplying the equation \eqref{sxE} by $u$ and integrating over
$(0,L)$ we obtain
$$
\frac12 \frac d{dt}\|P^{\frac12}u\|^2-\left( P^{\frac12}(u\px
u),P^{\frac12}u\right)+\|\px^2u\|^2 - \|\px u\|^2+3(u^2,(\px
u)^2)=0.
$$
Employing the inequality \eqref{Po1} and the interpolation
inequality
$$
\|\px v\|^2 \leq \eb_1\|P^{\frac12}u\| + C_1\|\px^2 v\|
$$
which is valid for each $v \in \dot H^2_{per}$, we obtain

\begin{multline}\label{sx4} \frac12\frac d{dt}\|P^{\frac12}u\|^2 +
(1-\eb_1)\|\px^2u\|^2+ 3(u^2,(\px u)^2)\leq
d_0\|u^2\|\|P^{\frac12}u\|+ C_1\|P^{\frac12}u\|^2 \le\\
\eb_2\|u^2\|^2+(C_1+C_2)\|P^{\frac12}u\|^2 \le 2 \eb_2\|u^2\|^2 +
\frac{[d_0(C_1+C_2)]^2}{2\eb_2}L.
\end{multline}
Finally we use the inequality (which follows from \eqref{Po1})
$$
\|u^2\|^2\leq 4d_0^2\left(u^2,(\px u)^2\right)
$$

and obtain from \eqref{sx4} the following estimate:
%$$
\be\label{sx5}
\frac12\frac d{dt}\|P^{\frac12}u\|^2 + (1-\eb_1)\|\px^2u\|^2+
(3-8\eb_2d_0^2)\left(u^2,(\px u)^2\right)\leq K_0.
\ee
%$$
 By choosing
$\eb_1=\frac12$ and $\eb_2=\frac3{8d_0^2}$ and applying the Gronwal inequality, we deduce \eqref{disdis} and finish the proof of the proposition.
\end{proof}
\begin{remark}\label{6reg} Using the standard parabolic regularity, one may show that the solution $u(t)$ constructed in Proposition \ref{abb1} becomes $C^\infty$ (and even Gevrey) regular for all $t>0$.
\end{remark}
\begin{remark}\label{6blow} By using the same arguments as in the proof of the Theorem \ref{gn}
we can show that a wide class of solutions to the initial boundary value problem for the sixth order  convective CH equations with concave potential
\be\label{usx1} \Dt
u -\px^4(\px^{2}u+u+u^3)+ u\px u=0.
 \ee
blow up in a finite time.\\
We would like also note that the blow up theorem for  sixth order unstable CH equations \eqref{usx1} {\it without} the convective term
can be established by using the concavity method of Levine since in this situation the equation can be written in the form \eqref{DO}.
\par
Finally, arguing as in the proof of Theorem \ref{Th.global}, one can show the global existence and dissipativity of solutions of the following problem:
\be\label{usx15} \Dt
u -\px^4(\px^{2}u+u+u|u|^p)+ u\px u=0
 \ee
 if $p<p_0$ for some exponent $p_0>0$, so the presence of the convective term prevents blow up in that situation as well. However, in order
 to compute the exponent $p_0$, we need the analog of the sharp Lemma \ref{Bronski} for the six order operator which we do not present here.
\end{remark}

%\subsection{Viscous Cahn - Hilliard equation with convective term} Here we consider the problem
%\be\label{vcch}\begin{cases}
% \Dt ( u -\nu
%\px^2u) +\px^4u+\px^{2}\(u -u^3\)+ \delta u\px u=0,\\
%u(0,t)=u(L,t)=\px^2u(0,t)=\px^2 u(L,t)=0, t>0,\\
%u(x,0)=u_0(x), x \in (0,L),
% \end{cases}
% \ee
%where $\nu>0, \delta\in \R$  are given numbers.
%In this case it also convinient to apply to both sides of the equation \eqref{vcch} the operator $P$
%and consider the equivalent problem
%\begin{equation}\label{vcchP}
%\begin{cases}
%P\Dt u+\nu \Dt u-\px^2 u-\delta P(u\px u)- u +u^3=0,\\
%u(0,t)=u(L,t)=0,\\
%u(x,0)=u_0(x).
%\end{cases}
%\end{equation}
%Arguing as in the proof of Proposition \ref{abb1} we see that solution of \eqref{vcchP} satisfies the inequality
%\be\label{sx5a}
%\frac12\frac d{dt}\left[\|P^{\frac12}u\|^2 +\nu \|u\|^2\right]+ (1-\eb_1)\|\px^2u\|^2+
%(3-8\eb_2d_0^2)\left(u^2,(\px u)^2\right)\leq K_0,
%\ee
%We choose here $\eb_1=\frac12,\eb_2=\frac{3}{8d_0^2}$ and deduce from \eqref{sx5a} the desired inequality
%$$
%\frac12\frac d{dt}\left[\|P^{\frac12}u\|^2 +\nu \|u\|^2\right]+m_0\left[\|P^{\frac12}u\|^2 +\nu \|u\|^2\right]\leq K_0,
%$$
%where $m_0=\frac{\lambda_1}{4}\min\{1,\lambda_1\}$ and $\lambda_1$ is the first eigenvalue of the Sturm-Liouville operator.
%The last estimate implies that the semigroup generated by the problem \eqref{vcch} has an absorbing ball in $L^2(0,L)$.

\end{document}